\documentclass{birkjour}
\usepackage{hyperref}

\usepackage{amssymb}
 \newtheorem{thm}{Theorem}[section]
 
 \newtheorem{lem}[thm]{Lemma}
 \newtheorem{prop}[thm]{Proposition}
 \theoremstyle{definition}
 \newtheorem{defn}[thm]{Definition}
 \theoremstyle{remark}
 \newtheorem{rem}[thm]{Remark}
 \newtheorem{ex}{Example}[section]
 \numberwithin{equation}{section}

\begin{document}

%
%
%
%
%
%
%
%
%

\title[Geometric properties of a novel type of orthogonality \ldots]
 {Geometric properties of a novel type of orthogonality via norm derivatives}

\author[Pal ]{ Kallal Pal}

\address{School of Mathematics,\\ Thapar Institute of Engineering and Technology,\\ Patiala-147004,  Punjab, India}

\email{kpal\_phd22@thapar.edu}

\thanks{The authors are grateful to the National Board of Higher Mathematics, Department of Atomic Energy, India (the research funding 02011/11/2020/NBHM
(PR)/R\&D-II/7830) to support this research work.}
\author[Chandok ]{ Sumit Chandok}
\address{School of Mathematics,\\ Thapar Institute of Engineering and Technology,\\ Patiala-147004,  Punjab, India}
\email{sumit.chandok@thapar.edu}


\date{ 2022}

\begin{abstract}

In this article, we generalize the notion of orthogonality as a linear combination of norm derivatives in order to give a novel concept that we refer to as $\rho_{\alpha,\beta}$-orthogonality.  Also, we discuss some of its geometric properties in a real normed linear space and present some sufficient criteria for the smoothness of a normed space by using $\rho_{\alpha,\beta}$-orthogonality. We provide a few examples to show that the $\rho_{\alpha,\beta}$- orthogonality cannot be compared to other well-known orthogonalities in any way. In addition to this, we offer a characterization of inner product spaces by making use of the functional notation $\rho_{\alpha,\beta}$. In addition, we show that any $\rho_{\alpha,\beta}$-orthogonality that preserves linear mapping between two normed linear spaces must necessarily be a scalar multiple of an isometry. Also, using the $\rho_{\alpha,\beta}$-functional, we define the idea of an angle between two vectors and talk about their characteristics in normed spaces.

\end{abstract}

\maketitle
\section{Introduction}

Orthogonality in normed spaces provides a powerful framework for understanding the geometry and structure of vector spaces equipped with norms, enabling the generalization of geometric concepts from Euclidean spaces to more abstract settings. While the Euclidean inner product corresponds to the dot product of vectors, orthogonality in normed spaces relies on the inner product and the norms of the vectors.
Norm derivatives are mentioned as tools for characterizing the geometric features of normed linear spaces. These derivatives arise naturally from the Gateaux derivative of the norm and are used to study concepts like strict convexity and smoothness.

The geometric structure of normed linear spaces can be developed using orthogonality relations derived from the norm derivatives.
The notion of orthogonality relation in a general normed linear space was introduced by Robert \cite{Robert1934orthogonalities}. Later in 1935, Birkhoff introduced Birkhoff orthogonality  \cite{birkhoff1935orthogonality}, which is one of the most important orthogonalities defined in normed linear space. James \cite{james1947orthogonality} provided a comprehensive study of the properties of Birkhoff orthogonality. Due to this, Birkhoff orthogonality has also been referred to Brikhoff-James orthogonality. Later, James \cite{10.1215/S0012-7094-45-01223-3} introduced the Isosceles orthogonality. Furthermore, James \cite{10.1215/S0012-7094-45-01223-3} introduced the Pythagorean orthogonality in normed linear space which generalizes the result in an Euclidean space stating that two vectors are perpendicular if and only if there is a right triangle having two vectors as legs. 


Amir (see \cite{amir1986characterizations}) introduced the functional using norm derivative, which has applications in studying the geometry of normed spaces. Many researchers have used the norm derivatives to characterize the fundamental geometric features of normed linear spaces, such as strict convexity and smoothness. The concepts of the norm derivatives arise naturally from the two-sided limiting nature of the Gateaux derivative of the norm, and therefore, they are proper generalizations of the latter. Later, Mili\'ci\'c \cite{milicic1987g} introduced a new mapping and its orthogonality as a combination of norm derivatives. Recently, Zamani and Moslehian \cite{zamani2019extension} introduced a new type of orthogonality in the context of normed spaces as a convex combination of norm derivatives. 

Also, orthogonality and angle relations are fundamental concepts in normed vector spaces, and their development is deeply rooted in the historical evolution of geometry and mathematics from Euclidean geometry to modern functional analysis.
Angles, angle functions, and the question of how to measure angles are all well-established ideas in mathematics with respect to Euclidean space; there are also several extensions to various non-Euclidean spaces. In particular, it is fascinating to explore the geometric features of several alternative notions of angle functions and angle measures in finite-dimensional real Banach spaces. For instance, several investigations in this field carried out in [see \cite{balestro2017angles},\cite{milicic2007b}], the authors established some sorts of orthogonality and some angle concept in normed spaces. In a normed space $X$, they also introduce the concept of an angle $P-,I-,$ and $g-$ angle between two vectors.

Inspired by these works, the purpose of this article is to introduce $\rho_{\alpha,\beta}$-orthogonality by generalizing some norm derivative type orthogonality as a linear combination. Also, we investigate its interesting properties using the norm derivatives from the perspective of two important geometric concepts, namely, orthogonality and smoothness. As a consequence, we obtain a relation between $\rho_{\alpha,\beta}$-orthogonality and Birkhoff-James orthogonality. We present some illustrative cases that demonstrate the incomparability of the $\rho_{\alpha,\beta}$-orthogonality with other well-known orthogonalities. Moreover, we demonstrate that every ${\rho_{\alpha,\beta}}$-orthogonality preserving linear mapping is necessarily a scalar multiple of an isometry. Also, we define the notion of the angle between two vectors using ${\rho_{\alpha,\beta}}$-functional and discuss their properties in normed space.

\section{Preliminaries}
This section defines some of the notations and terminologies that will be used throughout the article. 

Assume that $(X,\|.\|)$ is a real normed linear space of dimension greater than or equal to $2$. A vector $u \in X$ is said to be orthogonal to a vector $y\in X$ in the sense of Birkhoff–James (\cite{birkhoff1935orthogonality,james1947orthogonality}), written as
$u\perp_B v$, if \begin{align*}
    \|u+tv\|
\geq \|u\|,     
\text{~for all~} t \in \mathbb{R}.
\end{align*}
In 1986, Amir \cite{amir1986characterizations} defined the norm derivatives as, for all  $u, v \in X$, 
\[\rho_{\pm}(u,v)=\lim_{t \to 0^{\pm}}\frac{\|u+tv\|^2-\|u\|^2}{2t}=\|u\|\lim_{t \to 0^{\pm}}\frac{\|u+tv\|-\|u\|}{t}.\]

Mili\'ci\'c \cite{milicic1987g} defined the mapping $\rho : X\times X \rightarrow \mathbb{R}$ by\\
$$\rho(u,v)= \frac{\rho_{-}(u,v)+\rho_{+}(u,v)}{2},$$
for all $u,v \in X$ and introduced the corresponding $\rho$-orthogonality as follows: \\
$$u \perp_{\rho} v  
~\text{if and only if }~  
\rho(u,v)= \frac{\rho_{-}(u,v)+\rho_{+}(u,v)}{2}= 0.$$

Following are the well known properties of norm derivatives (see \cite{alsina2010norm}) in a real normed linear space $X$, for all  $u, v \in X$, we have
\begin{enumerate}
    \item[(i)]$
    \rho_{-}(u,v)\leq \rho_{+}(u,v)$;
    \item[(ii)] $ \beta\geq0$, $\rho_{\pm}(\beta u,v)=\beta \rho_{\pm}( u,v)=\rho_{\pm}( u,\beta v)$;
    \item [(iii)]  $ \beta < 0$, $\rho_{\pm}(\beta u,v)=\beta \rho_{\mp}( u,v)=\rho_{\pm}( u,\beta v)$;
    \item[(iv)] $\beta \in \mathbb{R}$,  $\rho_{\pm}(u,\beta u+v)=\beta \|u\|^2+\rho_{\pm}(u,v)$;
    \item[(v)]  $\rho_{\pm}(u,u)=\|u\|^2$;
    \item[(vi)]$|\rho_{\pm}(u,v)|\leq \|u\|\|v\|$;
\item[(vii)] $ u \perp_{\rho_{\pm}} v $;
~\text{if and only if }~ $\rho_{\pm}(u,v)=0.$
    
\end{enumerate}
It is interesting to note that the relations $\perp_{\rho_+}$,$\perp_{\rho_{-}}$ and $\perp_{\rho}$ are equivalent in an inner
product space but are generally incomparable in a normed space which is not smooth.\\
Furthermore, Zamani and Moslehian \cite{zamani2019extension} introduced $\rho_\lambda$-orthogonality  as an extension of orthogonality relations based on norm derivatives : 
$$ u \perp_{\rho_{\lambda}} v 
~\text{if and only if }~  
\rho_{\lambda}(u,v)= \lambda \rho_{-}(u,v)+(1-\lambda)\rho_{+}(u,v)= 0,$$
\text{for each }$u,v \in X$ \text{and} $\lambda \in[0,1]$.
Also, they gave a characterization of inner product spaces based on functional $\rho_\lambda$.\\
Due to  Lumer \cite{lumer1961semi} and  Giles \cite{giles1967classes} in every normed space $(X,\|.\|)$, there exists a mapping $[.,.] : X \times X \rightarrow \mathbb{C}$, known as a semi-inner product (s.i.p.), satisfying the following properties:
\begin{itemize}
 \item[(i)] $[\alpha u+\beta v,w]= \alpha [u,v]+\beta [v,z]$, for all $u,v,z \in u$ and $c,\beta \in \mathbb{C}$,
 \item[(ii)] $[u,\alpha v]=\Bar{\alpha}[u,v]$, for all $u,v \in X$ and $\alpha \in \mathbb{C}$,
 \item[(iii)] $[u,u]=\|u\|^2$, for all $u \in X$,
 \item[(iv)] $|[u,v]|\leq \|u\|\|v\|$, for all $u,v \in X$.
\end{itemize}

In an arbitrary normed space, the idea of orthogonality can be presented in a variety of ways.
 A semi-orthogonality of the components $u$ and $v$ in a semi inner product $[.,.]$ is defined by 
\begin{align*}
u \perp_{s} v \text{~if and only if~} [v,u]=0.
\end{align*}

The following results related to norm derivatives will be used in the sequel.

\begin{lem} \cite{alsina2010norm} \label{a}
 Let $(X,\|.\|)$ be a real normed linear space. Then the following conditions are equivalent:\\
 (i) $X$ is smooth.\\
 (ii) $\rho_{-}(u,v)=\rho_{+}(u,v)$, for all $u,v\in X$. 
\end{lem}
\begin{lem} \label{b}
[see \cite{alsina2010norm}, Proposition 2.1.7] 
 Let $(X,\|.\|)$ be a real normed linear space. Then $u \perp_{B} v$ if and only if $\rho_{-}(u,v) \leq 0 \leq \rho_{+}(u,v)$ for all $u,v \in X$. 
\end{lem}

\begin{lem} \label{c}
(see \cite{dragomir2004semi}) 
 Let $(X,\|.\|)$ be a real normed linear space with $u\neq 0$. Then there exists a number $t \in \mathbb{R}$ such that $u \perp_{B} tu+v$. 
\end{lem}
\begin{thm}\label{oo}
[see \cite{alsina2010norm}, Proposition 1.4.5]
 Let $(X,\|.\|)$ be a real normed linear space of dimension greater than or equal to $2$. The space $X$ is an inner product space if and only if each two dimensional subspace of $X$ is an inner product space.
\end{thm}
 
\section{Main Results}
First, we introduce an orthogonality relation as an extension of orthogonality relations based on norm derivatives $\rho_{\pm}$ denoted by $\rho_{\alpha,\beta}$. 

We define the notion of a  $\rho_{\alpha,\beta}$-orthogonality in a real normed linear space $(X,\|.\|)$:   
\[u \perp_{\rho_{\alpha,\beta}} v  ~\text{if and only if }~  \rho_{\alpha,\beta}(u,v)= \alpha \rho_{-}(u,v)+\beta\rho_{+}(u,v)= 0,\]  where $\alpha, \beta \in [0,1)$ with $0<\alpha + \beta < 1$ and $u,v\in X$.
\begin{rem}
\begin{enumerate}
\item[(a)]From the above definition, we have \[\perp_{\rho_{\alpha,\beta}}=
\begin{cases}
\perp_{\rho_{+}}, ~~ \text{if $\alpha=0, \beta\neq 0$}\\ 
\perp_{\rho_{-}},~~ \text{if $\beta=0, \alpha\neq 0$}\\
\perp_{\rho},~~~\text{if $\alpha\neq0\neq \beta$}.
\end{cases}\]
So, $\perp_{\rho_{\alpha,\beta}}$ is an extension orthogonality in normed space.\\
\item[(b)]   
Here, the $\rho_{\alpha,\beta}$-orthogonality class is more general than that of $\rho_\lambda$-orthogonality introduced by Zamani and Moslehian \cite{zamani2019extension}. \\ 
In the following example, we show that there are infinitely many vectors which satisfy $\perp_{\rho_{\alpha,\beta}}$ but not $\perp_{\rho_{\lambda}}$.
\begin{ex}
Suppose that $X=\mathbb{R}^2$ induced with a norm $\|(u,v)\|= \max\{|u|,|v|\}$. Let $u= (1,1), v= (-\frac{1}{2\alpha},\frac{1}{2\beta})\in X$, for $\alpha,\beta \in (0,1)$ with $\alpha+\beta<1$. Then  
$$\rho_{+}(u,v)=\frac{1}{2\beta},$$
$$\rho_{-}(u,v)=-\frac{1}{2\alpha},$$
$$\rho_{\alpha,\beta}(u,v)=(\alpha)(-\frac{1}{2\alpha})+(\beta)(\frac{1}{2\beta})=0.$$
So $u \perp_{\rho_{\alpha,\beta}} v$.\\

Here, $ \rho_{\lambda}(u,v)= \lambda \rho_{-}(u,v)+(1-\lambda)\rho_{+}(u,v)=\lambda(-\frac{1}{2\alpha})+(1-\lambda)(\frac{1}{2\beta})$.\\
Now, $u \perp_{\rho_\lambda} v$ holds if $\rho_\lambda(u,v)=0$. This is only possible when $\alpha=\lambda$ and $\beta=1-\lambda$, but by our definition $\alpha+\beta<1$. So we cannot find a $\lambda$ such that $\rho_\lambda(u,v)=0$. Therefore $u \not\perp_{\rho_\lambda} v$.

So, the class of orthogonality introduced in this paper is more general than that of Zamani and Moslehian \cite{zamani2019extension}.
\end{ex}

\end{enumerate}
\end{rem}

First, we give some properties of $\rho_{\alpha,\beta}$ based on norm derivatives $\rho_{\pm}$, which will be used in the sequel. 
\begin{prop}\label{k}
Let $(X,\|.\|)$ be a real normed linear space and $\rho_{\alpha,\beta}(u,v)= \alpha \rho_{-}(u,v)+\beta\rho_{+}(u,v),$  where $\alpha, \beta \in [0,1)$ with $0<\alpha + \beta < 1$. Then for all $u,v\in X$, we get\\
(i) $\rho_{\alpha,\beta}(u,u)=(\alpha+\beta)\|u\|^2<\|u\|^2$.\\
 (ii) $\rho_{\alpha,\beta}(tu,v)=t\rho_{\alpha,\beta}( u,v)=\rho_{\alpha,\beta}( u,tv)$, where $ t \geq 0$.\\
 (iii) $\rho_{\alpha,\beta}(tu,v)=t\rho_{\beta,\alpha}( u,v)=\rho_{\alpha,\beta}( u,tv)$, where $ t < 0$.\\
(iv) $\rho_{\alpha,\beta}(u,t u+v)=(\alpha+\beta)t\|u\|^2+\rho_{\alpha,\beta}(u,v)$, $t\in \mathbb{R}$.\\
(v) $|\rho_{\alpha,\beta}(u,v)|\leq(\alpha+\beta)\|u\|\|v\|<\|u\|\|v\|$.
\end{prop}

\begin{proof}
(i) For all $u\in X$, we have
\begin{eqnarray*}
\rho_{\alpha,\beta}(u,u)&=& \alpha \rho_{-}(u,u)+\beta\rho_{+}(u,u)\\
&=& \alpha \|u\|^2+\beta\|u\|^2\\
&=& (\alpha+\beta)\|u\|^2\\
&<& \|u\|^2.
\end{eqnarray*}
(ii) For all $u,v \in X$ and $ t \geq 0$, we have 
\begin{eqnarray*}
\rho_{\alpha,\beta}(tu,v) &=& \alpha \rho_{-}(tu,v)+\beta\rho_{+}(tu,v)\\
&=&  \alpha t \rho_{-}(u,v)+\beta t\rho_{+}(u,v)\\
&=& t(\alpha \rho_{-}(u,v)+\beta\rho_{+}(u,v))\\
&=& t\rho_{\alpha,\beta}(u,v).
\end{eqnarray*}
(iii) For all $u,v \in X$ and $ t < 0$, we have 
\begin{eqnarray*}
\rho_{\alpha,\beta}(tu,v) &=& \alpha \rho_{-}(tu,v)+\beta\rho_{+}(tu,v)\\
&=&  \alpha t \rho_{+}(u,v)+\beta t\rho_{-}(u,v)\\
&=& t(\alpha \rho_{+}(u,v)+\beta\rho_{-}(u,v))\\
&=& t\rho_{\beta,\alpha}(u,v).
\end{eqnarray*}
(iv) For all $u,v \in X$ and $t\in \mathbb{R}$, we have
\begin{eqnarray*}
\rho_{\alpha,\beta}(u,tu+v)&=&\alpha \rho_{-}(u,tu+v)+\beta\rho_{+}(u,tu+v)\\
&=& \alpha(t\|u\|^2+ \rho_{-}(u,v))+\beta(t\|u\|^2+\rho_{+}(u,v))\\
&=&(\alpha+\beta)t\|u\|^2+\alpha \rho_{-}(u,v)+\beta\rho_{+}(u,v)\\
&=&(\alpha+\beta)t\|u\|^2+\rho_{\alpha,\beta}(u,v)\\
&<& t\|u\|^2+\rho_{\alpha,\beta}(u,v).
\end{eqnarray*}
(v) For all $u,v\in X$, we have
\begin{eqnarray*}
|\rho_{\alpha,\beta}(u,v)|&=&|\alpha \rho_{-}(u,v)+\beta\rho_{+}(u,v)|\\
&\leq& \alpha |\rho_{-}(u,v)|+\beta|\rho_{+}(u,v)|\\
&\leq& \alpha\|u\|\|v\|+\beta\|u\|\|v\|\\
&\leq& (\alpha+\beta)\|u\|\|v\|\\
&<& \|u\|\|v\|.
\end{eqnarray*}

\end{proof}
\begin{thm}\label{thm3.2}
Let $(X,\|.\|)$ be a real normed linear space. Then 
$\perp_{\rho_{\alpha,\beta}} \subset \perp_{B}$.
\end{thm}
\begin{proof}
Suppose that $u\perp_{\rho_{\alpha,\beta}}v$ for all $u,v\in X$. As $\rho_{-}(u,v)\leq \rho_{+}(u,v)$,  we get 
\begin{align}\label{k1}
0&=\rho_{\alpha,\beta}(u,v)\notag\\
&=\alpha \rho_{-}(u,v)+\beta\rho_{+}(u,v)\notag\\
&\le \alpha \rho_{+}(u,v)+\beta\rho_{+}(u,v)\notag\\
&=(\alpha+\beta)\rho_{+}(u,v).
\end{align} 

It implies that $\rho_{+}(u,v)> 0$ as $\alpha+\beta<1$. 
Using the similar arguments, we get $\rho_{-}(u,v)< 0$. Therefore, $\rho_{-}(u,v) < 0 < \rho_{+}(u,v)$ and using Lemma \ref{b}, we get $u \perp_{B} v$.  Hence $\perp_{\rho_{\alpha,\beta}} \subset \perp_{B}$.
\end{proof}
The converse of the above theorem is not true in general. 
\begin{ex}
Let  $X=\mathbb{R}^2$ induced with a norm $\|(u,v)\|= \max\{|u|,|v|\}$. Choose  $u= (1,1), v= (1,-1)\in X$ and $\alpha\neq\beta$. Then $$ \|u\|=1,\|v\|=1,$$
$$\rho_{+}(u,v)=1,$$
$$\rho_{-}(u,v)=-1,$$
$$\rho_{\alpha,\beta}(u,v)=-\alpha+\beta,$$
for $\alpha,\beta \in (0,1)$ with $\alpha+\beta<1$.
Therefore, $\rho_{-}(u,v)\leq 0 \leq  \rho_{+}(u,v) $, and using Lemma \ref{b}, we get $u \perp_{B} v$.\\
Now, for $\alpha=\frac{1}{2}$ and $\beta=\frac{1}{3}$ , we have  $$\rho_{\alpha,\beta}(u,v)=
-\frac{1}{2}+\frac{1}{3} \neq 0.$$
So, $u \not\perp_{\alpha,\beta} v$.
\end{ex}
Smoothness is an important property that characterizes how well-behaved the norm is at each point in the space. It is well-known that $\perp_{B}=\perp_{\rho_{\pm}}$ and $\perp_{B}=\perp_{\rho}$ in a real normed linear space when $X$ is smooth. Now, we establish the following result as an analogue of these findings. 
\begin{thm}
 For $\alpha\neq\beta$, $\perp_{B} =\perp_{\rho_{\alpha,\beta}}$ if and only if $X$ is smooth space.
\end{thm}
\begin{proof}
Let $X$ be a smooth space and $u \perp_{B} v$. 
Then by using Lemma \ref{b}, we have $\rho_{-}(u,v)\leq 0 \leq  \rho_{+}(u,v)$. Using Lemma \ref{a} we have, if $X$ is smooth, then $\rho_{-}(u,v)=\rho_{+}(u,v)$. So, $\rho_{-}(u,v)=\rho_{+}(u,v)=0$. It implies that $\rho_{\alpha,\beta}(u,v)=\alpha \rho_{-}(u,v)+\beta\rho_{+}(u,v)= 0$.
So $u \perp_{\rho_{\alpha,\beta}} v$ and $\perp_{B} \subseteq \perp_{\rho_{\alpha,\beta}}$.
Therefore, by using Theorem \ref{thm3.2}, we get $\perp_{B} =\perp_{\rho_{\alpha,\beta}}$.\\
Conversely, let $\perp_{B} =\perp_{\rho_{\alpha,\beta}}$. We may assume that $u\neq 0$ otherwise $\rho_{-}(u,v)=\rho_+(u,v)$, that is $X$ is smooth. By Lemma \ref{c}, there exists a number $t \in \mathbb{R}$ such that $u \perp_{B} tu+v$. Therefore, $u \perp_{\rho_{\alpha,\beta}} tu+v$. Using the Proposition \ref{k}$(iv)$
we have,
\begin{align}\label{d}
(\alpha+\beta)t\|u\|^2+\rho_{\alpha,\beta}(u,v)=0.
\end{align}
Also, we have  $-u \perp_{B} tu+v$ and so $-u \perp_{\rho_{\alpha,\beta}} tu+v$. Using the Proposition \ref{k}$(iii),(iv)$
we get
\begin{align}\label{e}
-t(\alpha+\beta)\|u\|^2-\rho_{\beta,\alpha}(u,v)=0.
\end{align}
By adding \eqref{d} and \eqref{e}, we have
\begin{align*}
\rho_{\alpha,\beta}(u,v)&= \rho_{\beta,\alpha}(u,v).
\end{align*}
This result implies that $\rho_{+}(u,v)=\rho_{-}(u,v)$. Hence $X$ is smooth.
\end{proof}

Here it is interesting to note that the relation $\perp_{\rho_{-}}$, $\perp_{\rho{+}}$, $\perp_{\rho}$ and $\perp_{\rho_{\alpha,\beta}}$ are generally incomparable. 
In the following examples, we show that these orthogonalities are incomparable.
\begin{ex}
Consider a real normed linear space $\mathbb{R}^2$ with the norm $\|(u,v)\|= |u|+|v|$. Let $u= (1,0)$, $v= (1,1)$, $z=(-1,1)$ and $w=(0,2)$. Then it is easy to compute for $\alpha\neq\beta$ and $\alpha\neq0\neq \beta$ with $0<\alpha+\beta<1$\\

$\rho_{-}(u,v)=0$, $\rho_{+}(u,v)=2$, $\rho_{\alpha,\beta}(u,v)=2\beta,$\\

$\rho_{-}(u,z)=-2$, $\rho_{+}(u,z)=0$, $\rho_{\alpha,\beta}(u,z)=-2\alpha$,\\

and

$\rho_{-}(u,w)=-2$, $\rho_{+}(u,w)=2$, $\rho(u,w)=0$, $\rho_{\alpha,\beta}(u,w)=-2\alpha+2\beta $.\\
Therefore $\perp_{\rho_{-}} \nsubseteq\perp_{\rho_{\alpha,\beta}}$, $\perp_{\rho_{+}} \nsubseteq\perp_{\rho_{\alpha,\beta}}$, $\perp_{\rho} \nsubseteq\perp_{\rho_{\alpha,\beta}}$.
\end{ex}
\begin{ex}
Consider the real normed linear space $X=\mathbb{R}^2$ with the norm $\|(u,v)\|= \max\{|u|,|v|\}$. Let $u= (1,1)$, $v= (-\frac{1}{2\alpha},\frac{1}{2\beta})$, $\alpha\neq0\neq \beta$ with $0<\alpha+\beta<1$. Then, it is easy to compute
$$\rho_{+}(u,v)=\frac{1}{2\beta},$$
$$\rho_{-}(u,v)=-\frac{1}{2\alpha},$$
$$\rho_{\alpha,\beta}(u,v)=(\alpha)(-\frac{1}{2\alpha})+(\beta)(\frac{1}{2\beta})=0.$$
So $u \perp_{\rho_{\alpha,\beta}} v$.\\
Now, taking $\alpha=\frac{1}{2}$ and $\beta=\frac{1}{3}$ , we have  $$\rho_{+}(u,v)=\frac{2}{3} \neq 0,$$
$$\rho_{-}(u,v)=-1 \neq 0,$$
$$\rho(u,v)=\frac{\frac{2}{3}-1}{2} \neq 0.$$
Therefore, $\perp_{\rho_{\alpha,\beta}} \nsubseteq \perp_{\rho_{-}}$, $\perp_{\rho_{\alpha,\beta}} \nsubseteq \perp_{\rho_{+}}$, $\perp_{\rho_{\alpha,\beta}} \nsubseteq \perp_{\rho}$.
\end{ex}

Based on $\rho_{\alpha,\beta}$-orthogonality, the following result characterises smooth normed spaces.

\begin{thm}
For $\alpha\neq\beta$, $\perp_{s} =\perp_{\rho_{\alpha,\beta}}$ if and only if $X$ is smooth space
where $[., .]$ is a semi-inner
product in a normed space $X$ over a field C and $u, v \in X$. 
\end{thm}
\begin{proof}

 First assume that $\perp_{\rho_{\alpha,\beta}}= \perp_{s}$.
 Let $u,v \in X$. Taking $z=-\frac{{\rho_{\alpha,\beta}(u,v)}}{(\alpha+\beta)\|u\|^2}u+v$. Using Lemma \ref{k}(iv), we have $u\perp_{\rho_{\alpha,\beta}}z$. Therefore $u\perp_s z$. This implies that 
 \begin{align*}
     0=[z,u]=-\frac{{\rho_{\alpha,\beta}(u,v)}}{(\alpha+\beta)\|u\|^2}\|u\|^2+[v,u].
 \end{align*}
  Therefore ${\rho_{\alpha,\beta}}(u,v)={(\alpha+\beta)[v,u]}$. This result implies that $\rho_-(u,v)=\rho_+(u,v)$ and $X$ is smooth.\\
  Conversely let $X$ is smooth. Therefore $\rho_-(u,v)=\rho_+(u,v)=[v,u]$. By the definition of ${\rho_{\alpha,\beta}(u,v)}$, we get ${\rho_{\alpha,\beta}(u,v)}=(\alpha+\beta)[v,u]$. Hence $\perp_{\rho_{\alpha,\beta}}= \perp_{s}$.
  \end{proof}
\begin{thm}
Let $(X,\|.\|)$ be a normed space. For $\alpha\neq\beta$, the following conditions are equivalent:
\begin{itemize}
\item[(i)]$\perp_{\rho}=\perp_{\rho_{\alpha,\beta}}$.
\item[(ii)] $X$ is smooth. 
\end{itemize}
\end{thm}
\begin{proof}
First we prove that $(i)$ implies $(ii)$.\\
 Let $\perp_{\rho} =\perp_{\rho_{\alpha,\beta}}$. Then we have $u \perp_{\rho}(\frac{-\rho(u,v)}{\|u\|^2}u+v)$. Therefore, $u \perp_{\rho_{\alpha,\beta}} (\frac{-\rho(u,v)}{\|u\|^2}u+v)$ and so $\rho_{\alpha,\beta}(u,\frac{-\rho(u,v)}{\|u\|^2}u+v)=0$. Using the Proposition \ref{k}(iv) we have,
\begin{align*}\label{f}
 -(\alpha+\beta)\rho(u,v)+\rho_{\alpha,\beta}(u,v)=0.
\end{align*}
This result implies that $\rho_{+}(u,v)=\rho_{-}(u,v)$. Hence $X$ is smooth.\\
Suppose that $X$ is smooth and so, $\rho_{+}(u,v)=\rho_{-}(u,v)$. From the definition of $\rho$ and $\rho_{\alpha,\beta}$, it is clear that $\perp_{\rho}=\perp_{\rho_{\alpha,\beta}}$. Therefore $(ii)$ implies $(i)$.
\end{proof}
\begin{thm}
    Let $X$ be a nonempty set  endowed with two norms $\|.\|_1$ and $\|.\|_2$ over a field $\mathbb{R}$. $\|.\|_1$ and $\|.\|_2$ are norm equivalent if and only if there exists a positive constant $k$ such that, for all $u,v \in X$, 
    \begin{align*}
        |\rho_{\alpha,\beta,1}(u,v)-\rho_{\alpha,\beta,2}(u,v)|< k\min \{\|u\|_1\|v\|_1,\|u\|_2\|v\|_2\}
    \end{align*}
   where $\rho_{\alpha,\beta,c}$ is a functional $\rho_{\alpha,\beta}$ with respect to $\|.\|_c$, for $c=1,2$.
 \end{thm}
    

\begin{proof}
Suppose that $\|.\|_1$ and $\|.\|_2$ are norm equivalent. Then, there exist positive numbers $a,b$ such that $m\|.\|_1 \leq\|.\|_2\leq M\|.\|_1$.
Using Proposition \ref{k} (v), we have \[|\rho_{\alpha,\beta,1}(u,v)|< \|u\|_1\|v\|_1,\] and \[|\rho_{\alpha,\beta,2}(u,v)|< \|u\|_2\|v\|_2.\]
Therefore,
\begin{align*}
   |\rho_{\alpha,\beta,1}(u,v)-\rho_{\alpha,\beta,2}(u,v)|&< (\|u\|_1\|v\|_1+\|u\|_2\|v\|_2)\\
   &\leq (\|u\|_1\|v\|_1+M^2\|u\|_1\|v\|_1) \\
   &= (1+M^2)\|u\|_1\|v\|_1.
\end{align*}
On the similar lines, we get
\begin{align*}
 |\rho_{\alpha,\beta,1}(u,v)-\rho_{\alpha,\beta,2}(u,v)|< (1+\frac{1}{m^2})\|u\|_2\|v\|_2.   
\end{align*}
Taking $k=\max\{(1+M^2), (1+\frac{1}{m^2})\},$ we have,
\begin{align*}
|\rho_{\alpha,\beta,1}(u,v)-\rho_{\alpha,\beta,2}(u,v)|< k\min \{\|u\|_1\|v\|_1,\|u\|_2\|v\|_2\}.
 \end{align*}
Conversely, suppose that for every $u \in X$ and $k>0$, we have 
\begin{align*}
 |\rho_{\alpha,\beta,1}(u,u)-\rho_{\alpha,\beta,2}(u,u)|< k\min \{\|u\|^2_1,\|u\|^2_2\}.   
\end{align*}
If $\min \{\|u\|^2_1,\|u\|^2_2\}=\|u\|^2_1$, then
\begin{align*}
 |\|u\|^2_1-\|u\|^2_2| < k\|u\|^2_1 < k\|u\|^2_2.
\end{align*}
Therefore, we get 
\begin{align*}
\|u\|_2 < \sqrt{1+k}\|u\|_1,~~ \|u\|_1 < \sqrt{1+k}\|u\|_2,
\end{align*}
and  
\begin{align*}
\frac{1}{\sqrt{1+k}}\|u\|_1 < \|u\|_2 < \sqrt{1+k}\|u\|_1.
\end{align*}
Similarly, we get the result, if $\min \{\|u\|^2_1,\|u\|^2_2\}=\|u\|^2_2$.
\end{proof}
In real inner product space X, it is evident that the equality,
\begin{align}\label{a28}
\|u+v\|^4-\|u-v\|^4=8(\|u\|^2\langle u,v\rangle+\|v\|^2\langle v,u\rangle)
\end{align}
holds, which is equivalent to the parallelogram equality
\begin{align*}
\|u+v\|^2-\|u-v\|^2=2(\|u\|^2+\|v\|^2).
\end{align*}
In normed spaces, the equality
\begin{align*}
(\alpha+\beta)(\|u+v\|^4-\|u-v\|^4)=8(\|u\|^2\rho_{\alpha,\beta}(u,v)+\|v\|^2\rho_{\alpha,\beta}(v,u))
\end{align*}
is a generalization of \eqref{a28}, for $\alpha,\beta\in[0,1)$ with $0<\alpha+\beta<1$.\\
In the following result, we provide a sufficient condition for the smoothness of a normed space. \\
Also we know that the function $\rho_\pm$ is continuous in the second variable (see \cite{alsina2010norm}, Proposition 2.1.3). Therefore, functional $\rho_{\alpha,\beta}$ is continuous.
\begin{thm}
Suppose that $(X, \|.\|)$ is a normed space  and
\begin{align}\label{a29}
(\alpha+\beta)(\|u+v\|^4-\|u-v\|^4)=8(\|u\|^2\rho_{\alpha,\beta}(u,v)+\|v\|^2\rho_{\alpha,\beta}(v,u)),
\end{align}
holds for $u, v\in X$, $\alpha,\beta \in (0,1)$ such that $\alpha+\beta<1$. Then $X$ is smooth.
\end{thm}
\begin{proof}
For $u, v\in X$, $\alpha,\beta \in (0,1)$ such that $\alpha+\beta<1$, consider
\begin{align*}
(\alpha+\beta)&(\|u+v\|^4-\|u-v\|^4)\\
&=(\alpha+\beta) \lim_{t \to 0^+} \left(\|(u+\frac{t}{2}v)+v\|^4-\|(u+\frac{t}{2}v)-v\|^4\right)\\
&= 8 \lim_{t \to 0^+}\left[\|u+\frac{t}{2}v\|^2\rho_{\alpha,\beta}\left(u+\frac{t}{2}v,v\right)+\|v\|^2\rho_{\alpha,\beta}\left(v,u+\frac{t}{2}v\right)\right]\\
&~~~~~\hspace{7cm}[\text{Using \eqref{a29}}]\\
&= 8 \lim_{t \to 0^+} \left[\|u+\frac{t}{2}v\|^2\rho_{\alpha,\beta}\left(u+\frac{t}{2}v,v\right)+\|v\|^2\left(\frac{t}{2}(\alpha+\beta)\|v\|^2+\rho_{\alpha,\beta}(v,u)\right)\right]\\
&~~~~~\hspace{6.5cm}[\text{Using proposition \ref{k}(iv)}]\\
&= 8 \left[\|u\|^2\lim_{t \to 0^+}\rho_{\alpha,\beta}\left(u+\frac{t}{2}v,v\right)+\|v\|^2\rho_{\alpha,\beta}(v,u)\right].
\end{align*}
Therefore, by using \eqref{a29}, we have
\begin{align}\label{a30}
\lim_{t \to 0^+}\rho_{\alpha,\beta}\left(u+\frac{t}{2}v,v\right)=\rho_{\alpha,\beta}(u,v).    
\end{align}
For $u, v\in X$, consider
\begin{align*}
 \rho_{+}&(u,v)=\lim_{t \to 0^+}\frac{\|u+tv\|^2-\|u\|^2}{2t}\\[0.2cm] 
 &=\lim_{t \to 0^+}\frac{\|u+tv\|^4-\|u\|^4}{2t(\|(u+tv\|^2+\|u\|^2)}\\[0.2cm]
&=\lim_{t \to 0^+}\frac{\|(u+\frac{t}{2}v)+\frac{t}{2}v\|^4-\|(u-\frac{t}{2}v)+\frac{t}{2}v\|^4}{2t(\|(u+tv\|^2+\|u\|^2)}\\[0.2cm]
&=\lim_{t \to 0^+}\frac{8\left(\|u+\frac{t}{2}v\|^2\rho_{\alpha,\beta}\left(u+\frac{t}{2}v,\frac{t}{2}v\right)+\|\frac{t}{2}v\|^2\rho_{\alpha,\beta}\left(\frac{t}{2}v,u+\frac{t}{2}v\right)\right)}{2t(\|(u+tv\|^2+\|u\|^2)(\alpha+\beta)}\\[0.2cm]
&~~~~~\hspace{6.5cm}[\text{Using \eqref{a29} and \eqref{a30}}]\\
&=\lim_{t \to 0^+}\frac{4t\left(\|u+\frac{t}{2}v\|^2\rho_{\alpha,\beta}\left(u+\frac{t}{2}v,v\right)+\|\frac{t}{2}v\|^2\rho_{\alpha,\beta}\left(v,u+\frac{t}{2}v\right)\right)}{2t(\|(u+tv\|^2+\|u\|^2)(\alpha+\beta)}\\[0.2cm]
&~~~~~\hspace{6cm}[\text{Using proposition \ref{k}(iii)}]\\
&=2\lim_{t \to 0^+}\frac{\|u+\frac{t}{2}v\|^2\rho_{\alpha,\beta}\left(u+\frac{t}{2}v,v\right)+\frac{t^3}{8}\|v\|^4+\frac{t^2}{4}\|v\|\rho_{\alpha,\beta}\left(v,u\right)}{(\|(u+tv\|^2+\|u\|^2)(\alpha+\beta)}\\[0.2cm]
&~~~~~\hspace{6cm}[\text{Using proposition \ref{k}(iv)}]\\
&=2\times\frac{\|u\|^2 \rho_{\alpha,\beta}(u,v)+0}{2\|u\|^2(\alpha+\beta)}\\
&=\frac{\rho_{\alpha,\beta}(u,v)}{(\alpha+\beta)}.
\end{align*}
It implies that $\rho_{+}(u,v)=\rho_{-}(u,v)$. Hence, $X$ is smooth. 
\end{proof}


This section concludes with another characterization of inner product spaces if $\rho_{\alpha,\beta}$ is symmetrical.
\begin{thm}
 $\rho_{\alpha,\beta}(u,v)=\rho_{\alpha,\beta}(v,u)$ if and only if the norm in $X$ comes from a real inner product. 
\end{thm}
\begin{proof}
Let the norm in $X$ come from a real inner product space. Then $\rho_{-}(u,v)=\langle u,v\rangle=\rho_{-}(u,v)$. Therefore, using the definition of $\rho_{\alpha,\beta}$, we have 
\[\rho_{\alpha,\beta}(u,v)=(\alpha+\beta)\langle u,v\rangle.\] By the property of real inner product $\langle u,v\rangle=\langle v,u\rangle$, we have  \[\rho_{\alpha,\beta}(u,v)=\rho_{\alpha,\beta}(v,u).\]

Conversely, let  $\rho_{\alpha,\beta}(u,v)=\rho_{\alpha,\beta}(v,u)$.
This condition implies that $\rho_{\beta,\alpha}(u,v)=\rho_{\beta,\alpha}(v,u)$ for all $u,v \in X$.
Using Lemma \ref{k}(ii), we have
$\rho_{\beta,\alpha}(u,v)=-\rho_{\alpha,\beta}(-u,v)=\rho_{\alpha,\beta}(v,-u)=\rho_{\beta,\alpha}(v,u)$.\\
Now let $P$ be any two-dimensional subspace of $X$. Define mapping $\langle .,. \rangle:P\times P\rightarrow \mathbb{R}$ by 
\begin{align*}
    \langle u,v\rangle = \frac{\rho_{\alpha,\beta}(u,v)+\rho_{\beta,\alpha}(u,v)}{2(\alpha+\beta)}.
\end{align*}
We demonstrate that ${\langle .,. \rangle}$ is a inner product in $P$.\\
(i) Using Proposition \ref{k}(i), we have  $\rho_{\alpha,\beta}(u,u)=(\alpha
+\beta)\|u\|^2\geq 0$. Also $\rho_{\beta,\alpha}(u,u)\geq 0$. Therefore $\langle u,u \rangle  \geq 0$. Also when $u=0$, we have $\langle u,v\rangle=0$.\\
(ii) By the condition $\rho_{\alpha,\beta}(u,v)=\rho_{\alpha,\beta}(v,u)$ and $\rho_{\beta,\alpha}(u,v)=\rho_{\beta,\alpha}(v,u)$. Therefore $\langle u,v \rangle=\langle v,u \rangle$.\\
(iii) For $\alpha\geq0$ we have
\begin{align*}
     \langle \alpha u,v \rangle &= \frac{\rho_{\alpha,\beta}(\alpha u,v)+\rho_{\beta,\alpha}(\alpha u,v)}{2(\alpha+\beta)}\\
     &= \frac{\alpha\rho_{\alpha,\beta}(u,v)+\alpha\rho_{\beta,\alpha}(u,v)}{2(\alpha+\beta)}\\
     &= \alpha \langle u,v \rangle .
\end{align*}
 For $\alpha<0$ we have
\begin{align*}
     \langle \alpha u,v \rangle &= \frac{\rho_{\alpha,\beta}(\alpha u,v)+\rho_{\beta,\alpha}(\alpha u,v)}{2(\alpha+\beta)}\\
     &= \frac{\alpha\rho_{\beta,\alpha}(u,v)+\alpha\rho_{\alpha,\beta}(u,v)}{2(\alpha+\beta)}\\
     &= \alpha \langle u,v \rangle .
\end{align*}
(iv) It is sufficient to demonstrate the additivity with regard to the second variable. Take $u,v,w \in X$.\\
Case 1: $u$ and $v$ are linearly dependent that is $v=tu$ for some $t \in \mathbb{R}$. Now,
\begin{align*}
    \langle u,v+w \rangle &=\langle u,tu+w\rangle\\
    &= \frac{\rho_{\alpha,\beta}(u,tu+w)+\rho_{\beta,\alpha}(u,tu+w)}{2(\alpha+\beta)}\\
    &= \frac{t(\alpha+\beta)\|u\|^2+\rho_{\alpha,\beta}(u,w)+t(\alpha+\beta)\|u\|^2+\rho_{\beta,\alpha}(u,w)}{2(\alpha+\beta)}\\
    &=t\|u\|^2+\frac{\rho_{\alpha,\beta}(u,w)+\rho_{\beta,\alpha}(u,w)}{2(\alpha+\beta)}\\
    &=\langle u,tu \rangle +\langle u,w \rangle\\
    &=\langle u,v \rangle +\langle u,w \rangle.
\end{align*}
Case 2: $u$ and $v$ are linearly independent that is $w=s_1u+s_2v$ for some $s_1, s_2 \in \mathbb{R}$. Now,
\begin{align*}
    \langle u,v+w \rangle &=\langle u,s_1u+(1+s_2)v\rangle\\
    &= \frac{\rho_{\alpha,\beta}(u,s_1u+(1+s_2)v)+\rho_{\beta,\alpha}(u,s_1u+(1+s_2)v)}{2(\alpha+\beta)}\\
    &= \frac{2s_1(\alpha+\beta)\|u\|^2+\rho_{\alpha,\beta}(u,(1+s_2)v)+\rho_{\beta,\alpha}(u,(1+s_2)v)}{2(\alpha+\beta)}\\
    &=s_1\|u\|^2+\frac{\rho_{\alpha,\beta}(u,(1+s_2)v)+\rho_{\beta,\alpha}(u,(1+s_2)v)}{2(\alpha+\beta)}\\
    &=\langle u,s_1u \rangle +\langle u,(1+s_2)v \rangle\\
    &=\langle u,s_1u \rangle +(1+s_2)\langle u,v \rangle\\
    &=\langle u,s_1u \rangle+\langle u,v \rangle+ \langle u,s_2v \rangle\\
    &= \langle u,v \rangle+\langle u,s_1u+s_2v \rangle\\
    &= \langle u,v \rangle+\langle u,w\rangle.
\end{align*}
Thus $\langle .,.\rangle$ is an inner product in $P$. Using the Theorem \ref{oo}, we get the norm in $X$ comes from an inner
product.
\end{proof}

\subsection{Linear mappings preserving $\rho_{\alpha,\beta}$-orthogonality}

 Many researchers have examined the linear preserver problems, that is, to determine the structure of linear mappings between normed spaces that preserve linear orthogonality.

We know that a support functional $F_u$ at $u \in X$ is a norm-one linear functional
in $X^*$ such that $F_u(u) =\|u\|$. The Hahn-Banach theorem states that for every $u\in X$, there is always at least one such functional. If there is a distinct support functional at $u$, a normed space $X$ is said to be smooth at $u \in X\backslash \{0\}$.
It is important to note that the norm $\|.\|$ is said to be Gateaux differential at $u \in X$ if the
limit $f_u(v)=\lim_{t \to 0}\frac{\|u+tv\|-\|u\|}{t}$
exists for all $v\in X$. The Gateaux differential at $u$ of the norm $\|.\|$ is denoted as $f_u$. Also, $f_u$ is a bounded real linear functional on $X$. When $u$ is a smooth point, it is evident that $\rho_+(u,v)=\rho_-(u,v)=\|u\|f_u(v)$ for $v\in X$.
Hence, the smoothness of $X$ at $u$ is equivalent to the Gateaux differentiability of the norm at $u$.

We quote the following well known result for further reference. 

\begin{lem}\label{h}(see \cite{blanco2006maps})
	Every norm on $\mathbb{R}^n$ is Gateaux differentiable $\mu^n$-a.e. on $\mathbb{R}^n$, where $\mu^n$ is a Lebesgue measure on $\mathbb{R}^n$.
\end{lem}

The following result of Blanco and Turn{\v{s}}ek \cite{blanco2006maps} will be used in the sequel. 
\begin{lem}\label{j}
    Let $\|.\|$ be any norm on $\mathbb{K}^2$, where $\mathbb{K}$ is real or complex field and let $D \subseteq \mathbb{K}^2$ be a set of all non-smooth points. If $\mu^{2 \dim \mathbb{K}}(D)=0$, then there exist a path $g:[0,2]\rightarrow \mathbb{K}^2$ of the form:
    \begin{align*}
        g(s):=
        \begin{cases}
            &(1,s\xi), s \in [0,1]\\
            &(1,(2-s)\xi+(s-1)), s\in [1,2]\\
        \end{cases}
    \end{align*}
    for some $\xi \in \mathbb{K}$, such that $\mu\{s:g(s)\in D\}=0$.
\end{lem}
\begin{lem} \label{pp}
 If a real normed space $X$ is smooth at
$u\in X \backslash \{0\}$ and $F_u$ is the unique support functional at $u$, then the following
are equivalent, for every $v \in X$.\\
(i) $\rho_{\alpha,\beta}(u,v)=0$,\\
(ii) $v\in ker F_u$.
\end{lem}
\begin{proof}
Here $X$ is smooth, so there exists unique support functional $F_u$ such that $\rho_+(u,v)=\rho_-(u,v)=\|x\|F_u(v)$. \\
Let us assume that $(ii)$ holds. Since $v\in ker F_u$, $\rho_+(u,v)=\rho_-(u,v)=0$. Therefore $\rho_{\alpha,\beta}(u,v)=0$. \\
Suppose that (i) holds. Then $\rho_{\alpha,\beta}(u,v)=(\alpha+\beta)F_u(y)=0$. This implies that $v\in ker F_u$.
\end{proof}
Now, we are in position to formulate main result of the section.

\begin{thm} \label{o}
Let $X$ and $\mathcal{Y}$ be normed spaces and let $\mathcal{T}:X \rightarrow \mathcal{Y}$ be a non-zero bounded linear operator. Then the following conditions are equivalent:
\begin{itemize}
\item[(i)] $\mathcal{T}$ preserves  $\rho_{\alpha,\beta}$-orthogonality, that is, if $u \perp_{\rho_{\alpha,\beta}} v$ then $\mathcal{T}u \perp_{\rho_{\alpha,\beta}} \mathcal{T}v$,
\item[(ii)] $\|\mathcal{T}u\|=\|\mathcal{T}\|\|u\|$, for all $u \in X$,
\item[(iii)] $\rho_{\alpha,\beta}(\mathcal{T}u,\mathcal{T}v)=\|\mathcal{T}\|^2\rho_{\alpha\beta}(u,v)$, for all $u,v \in X$.
\end{itemize}

\end{thm}
\begin{proof}
First we prove that $(i)$ implies $(ii)$.
Assume that $\mathcal{T}$ preserves  $\rho_{\alpha,\beta}$-orthogonality. To prove that $\mathcal{T}$ is a scalar multiple of isometry, we show that\\
(a) $\mathcal{T}$ is injective; that is $\mathcal{T}u=0$ implies $u=0$;\\
(b) $\mathcal{T}$ is an isometry; that is $\|u\|=\|v\|$ implies $\|\mathcal{T}u\|=\|\mathcal{T}v\|$.\\
To prove (a), suppose that $\mathcal{T}u=0$, for some $u \neq 0$. Assume that $v$ is a component of $X$ that is independent to $u$. Then we can select a number $n \in \mathbb{N}$ such a way that $\frac{\|v\|(\alpha+\beta)}{n\|u+\frac{1}{2n}v\|}<1$. Taking $w=u+\frac{1}{2n}v$ and using Proposition \ref{k}(v) we have
\begin{align} \label{p}
    0< 1-\frac{\|v\|(\alpha+\beta)}{n\|w\|}=1-\frac{\|v\|\|w\|(\alpha+\beta)}{n\|w\|^2} \leq 
    1-\frac{\rho_{\alpha,\beta}(w,v)}{n\|w\|^2}.
\end{align}
Also, $\rho_{\alpha,\beta}(w,-\frac{{\rho_{\alpha,\beta}}(w,v)}{\|w\|^2(\alpha+\beta)}w+v)=0$. Since $\mathcal{T}$ preserves 
$\rho_{\alpha,\beta}$-orthogonality, so $\rho_{\alpha,\beta}(\mathcal{T}w,-\frac{{\rho_{\alpha,\beta}}(w,v)}{\|w\|^2(\alpha+\beta)}\mathcal{T}w+\mathcal{T}v)=0$. Since $Tu = 0$, by Proposition \ref{k}(iv) we have that
\begin{align} \label{q}
\nonumber0 &=\rho_{\alpha,\beta}(\mathcal{T}w,-\frac{{\rho_{\alpha,\beta}}(w,v)}{\|w\|^2(\alpha+\beta)}\mathcal{T}w+\mathcal{T}v)\\
\nonumber&=\rho_{\alpha,\beta}\left(\frac{1}{n}\mathcal{T}v,-\frac{{\rho_{\alpha,\beta}}(w,v)}{\|w\|^2(\alpha+\beta)}\frac{1}{n}\mathcal{T}v+\mathcal{T}v\right)\\
&=\frac{1}{n}\left(1-\frac{\rho_{\alpha,\beta}(w,v)}{n\|w\|^2}\right)\|\mathcal{T}v\|^2.
\end{align}
From \eqref{p} and \eqref{q}, we have $\mathcal{T}v=0$ for all $v$
independent of $u$.
Hence $\mathcal{T}=0$, a contradiction. Therefore  $\mathcal{T}$ is injective.\\
Now, we have to prove that $\|u\|=\|v\|$ implies $\|\mathcal{T}u\|=\|\mathcal{T}v\|$. 

It holds easily if $u$ and $v$ are linearly dependent, for then
$\|u\|=\|v\|$ implies $v=ru$ for some scalar $r$ of modulus $1$. So, $\|\mathcal{T}u\|=\|\mathcal{T}v\|$.

Let $u$ and $v$ be linearly independent.
Assume that $E$ is a linear subspace of $X$ generated by $u,v$. Define $\|x\|_\mathcal{T}:=\|\mathcal{T}x\|$ for $x \in E$. Here $\|.\|$ is a norm on $E$ because $\mathcal{T}$ is injective. Let $\mathcal{V}$ represent the set of points $x \in E$ at which at least one of the norms, $\|.\|$ or $\|.\|_T$, is not Gateaux differentiable. Let $F_x$ and $G_x$ be unique support functionals at $x$ with respect to $\|.\|$ and $\|.\|_\mathcal{T}$ respectively  for $x \in E \backslash \mathcal{V}$.
Let $y \in kerF_x$. Since $(E,\|.\|)$ is smooth at $x$, using Lemma \ref{pp}, we have $\rho_{\alpha,\beta}(x,y)=0$.
Hence $\rho_{\alpha,\beta}(\mathcal{T}x,\mathcal{T}y)=0$. Since $(E,\|.\|_\mathcal{T})$ is smooth at $x$, we get $G_x(v)=\frac{1}{\|Tu\|}\rho_{\alpha,\beta}(\mathcal{T}x,\mathcal{T}y)=0$. Therefore, $y \in kerG_x$ and $ker F_x \subseteq ker G_x$. 
Then there exists a function
$\lambda:E\backslash \mathcal{V} \rightarrow \mathbb{R}$ such that $G_x=\lambda(x)F_x$ for all $x \in E\backslash \mathcal{V}$. As
\begin{align*}
\|\mathcal{T}x\|=G_x(x)=\lambda(x)F_x(x)=\lambda(x)\|x\|, ~ x \in E\backslash \mathcal{V}, 
\end{align*}
it is easily observed that $\lambda$ is in fact real-valued.

Let $ \mathcal{L} : \mathbb{R}^2\rightarrow E$ defined by $(c,d) \mapsto cu+d(v-u)$. It is obvious that, $\mathcal{L}$ is a linear isomorphism.
Set $D=\mathcal{L}^{-1}(\mathcal{V})$. From the definition of $\mathcal{V}$, it is clear that $D$ is the collection of those points $(c,d) \in \mathbb{C}^2$ at which at least one of the functions $(c,d) \mapsto \|\mathcal{L}(c,d)\|$ or  $(c,d) \mapsto \|\mathcal{L}(c,d)\|_\mathcal{T}$ is not Gateaux differentiable. As both these functions are norms in $\mathbb{R}^2$,
so by the 
Lemma \ref{h},  $\mu^2(D) = 0$. Therefore
$g:[0,2]\rightarrow \mathbb{R}^2$ be the path obtained in Lemma \ref{j}. 
Then $\Phi:[0, 2] \rightarrow E$ is defined by
\begin{align*}
 \Phi(s):=\frac{\|u\|}{\|\mathcal{L}(g(s))\|}\mathcal{L}(g(s)) , ~~s\in [0,2]  
\end{align*}
and $\mu\{s :\Phi(s)\in \mathcal{V}\}=\mu\{s:g(s)\in D\}=0$. \\
Here, 
\begin{align*}
\mathcal{L}(g(s))=
 \begin{cases}
            &u+s\xi(v-u), ~s \in [0,1]\\
            &u+((2-s)\xi+(s-1))(v-u),~s\in [1,2]\\
\end{cases}
\end{align*}
Now, $s_1,s_2 \in [0,1]$ we have
\begin{align*}
 |\|\mathcal{L}(g(s_1))\|-\|\mathcal{L}(g(s_2))\|| \leq |\xi||s_1-s_2|\|v-u\|.  
\end{align*}
If $s_1,s_2 \in [1,2]$, we have 
\begin{align*}
 |\|\mathcal{L}(g(s_1))\|-\|\mathcal{L}(g(s_2))\|| \leq |1-\xi||s_1-s_2|\|v-u\|.  
\end{align*}
Finally, if $s_1 \in [0,1]$ and $s_2 \in [1,2]$, then
\begin{align*}
 |\|\mathcal{L}(g(s_1))\|-\|\mathcal{L}(g(s_2))\|| \leq (1+|\xi|)|s_1-s_2|\|v-u\|.  
\end{align*}
Therefore $s \mapsto \|Lg(s)\|$ satisfy Lipschitz conditions. Similarly  $s \mapsto \|\mathcal{L}g(s)\|_T$ satisfy Lipschitz conditions. It follows that
\begin{align*}
\|\Phi(s)\|_\mathcal{T}:=\frac{\|u\|\|\mathcal{L}(g(s))\|_\mathcal{T} }{\|\mathcal{L}(g(s))\|}   
\end{align*}
is absolutely continuous and that
\begin{align*}
\mu\{t:\Phi^{'}(s) ~\textit{does not exist}\}=\mu\{s:\|\mathcal{L}g(s)\|^{'} ~\textit{does not exist}\}=0.  
\end{align*}
Also note that $s \mapsto \|\Phi(s)\|=\|u\|$ is a constant function. We then have
$\|\Phi(s)\|^{'}_T=0$ $\mu$-a.e. on $[0,2]$. Hence, $s \mapsto \|\phi(s)\|_\mathcal{T}$  is a constant function, and we have that $\|u\|_\mathcal{T}=\|\Phi(0)\|_\mathcal{T}=\|\Phi(2)\|_\mathcal{T}=\|v\|_\mathcal{T}$. Hence $\|\mathcal{T}u\|=\|\mathcal{T}v\|$.\\
Therefore $T$ is an isometry.

 We prove that $(ii)$ implies $(iii)$. Suppose $(ii)$ holds. So $\|\mathcal{T}u\|=\|\mathcal{T}\|\|u\|$, for all $u \in X$.
Since $T$ is bounded linear, we have 
\begin{align*}
\rho_{+}(Tu,Tv)=&\|Tu\|\lim_{t \to 0^+}\frac{\|Tu+t Tv\|-\|Tu\|}{t}\\
=& \|T\|\|u\|\lim_{t \to 0^+}\frac{\|T\|\|u+t v\|-\|T\|\|u\|}{t}\\
=& \|T\|^2\|u\|\lim_{t \to 0^+}\frac{\|u+t v\|-\|u\|}{t}\\
=& \|T\|^2\rho_{+}(u,v).
\end{align*} 
Similarly, $\rho_{-}(u,v)=\|T\|^2\rho_{-}(u,v)$. From the definition of $\rho_{\alpha,\beta}$, we have $\rho_{\alpha,\beta}(\mathcal{T}u,\mathcal{T}v)=\|\mathcal{T}\|^2\rho_{\alpha,\beta}(u,v)$, for all $u,v \in X$.\\ 
Next we show that $(iii)$ implies $(i)$. Let $(iii)$ holds. Then $\rho_{\alpha,\beta}(\mathcal{T}u,\mathcal{T}v)=\|\mathcal{T}\|^2\rho_{\alpha,\beta}(u,v)$, for all $u,v \in X$. If $\rho_{\alpha,\beta}(u,v)=0$, then $\rho_{\alpha,\beta}(\mathcal{T}u,\mathcal{T}v)=0$. Therefore  $\mathcal{T}$ preserves  $\rho_{\alpha,\beta}$-orthogonality.

\end{proof}
\subsection{$\alpha,\beta$-Angular Norms}
From a geometrical perspective, the idea of an angle and the question of how to measure angles is always intriguing. In this section, we look at a specific kind of angle function that depends on $\rho_{\alpha, \beta}$ functional. \\
In a real inner product space $(X. \langle .,.\rangle)$, the angle $\theta(u,v)$ between two non-zero elements $u, v$ is defined by 
\begin{align*}
 \theta(u, v)= \arccos \frac{\langle u,v\rangle}{\|u\|\|v\|}   
\end{align*}
\begin{defn}\label{ac}
The number
\begin{align*}
ang_{\alpha,\beta}(u,v)=\theta_{\alpha,\beta}(u, v)= \arccos \frac{\rho_{\alpha,\beta}(u, v)}{(\alpha+\beta)\|u\|\|v\|}    
\end{align*}
is called the $\rho_{\alpha,\beta}$-angle between the element $u$ and the element $v$ in a normed linear space.    
\end{defn}
\begin{rem}
\begin{itemize}
    \item[(i)] From the Proposition \ref{k}$(v)$ we have $-1\leq \frac{\rho_{\alpha,\beta}(u,v)}{(\alpha+\beta)\|u\|\|v\|}\leq 1$, so the definition of angle is well-define.
    \item[(ii)] If the norm in $X$ arises from an inner product, it is easy to see that $\rho_{\alpha,\beta}$-angle agree with the angle defined by the inner product.
\end{itemize}
\end{rem}
\begin{prop}
$\rho_{\alpha,\beta}$-angle satisfies the following properties:
\begin{itemize}
\item[(a)] If $u$ and $v$ are of the same directions, then $\theta_{\alpha,\beta}(u, v)=0$, and if $u$ and $v$ are of 
opposite directions, then $\theta_{\alpha,\beta}(u, v)=\pi$ (part of parallelism property). 
\item[(b)]  $\theta_{\alpha,\beta}(au, bv)=\begin{cases}
\theta_{\alpha,\beta}(u, v),~\text{ if~} ab > 0;\\
\pi-\theta_{\beta, \alpha}(u, v),~\text{ if~} ab < 0
\end{cases}$ (homogeneity property);
\end{itemize}
\end{prop}
\begin{proof}
Using the Proposition \ref{k} and Definition \ref{ac}, we have
\begin{itemize}
\item[(a)] If $v = tu$, then
\begin{align*}
\theta_{\alpha,\beta}(u,v)
=&\theta_{\alpha,\beta}(u, tu)\\
=& \arccos \frac{\rho_{\alpha,\beta}(u, tu)}{(\alpha+\beta)\|u\|\|tu\|}\\
=&\arccos \frac{t(\alpha+\beta)\|u\|^2}{|t|(\alpha+\beta)\|u\|^2}\\
=&\begin{cases}
\arccos (1),~\text{ if~} t > 0;\\
\arccos (-1),~\text{ if~} t < 0
\end{cases}\\
=&\begin{cases}
0,~\text{ if~} t > 0;\\
\pi,~\text{ if~} t < 0.
\end{cases}
\end{align*}

\item[(b)] For $a,b\in \mathbb{R}$, we have 
\begin{align*}
\theta_{\alpha,\beta}(au,bv)
=& \arccos \frac{\rho_{\alpha,\beta}(au,bv)}{(\alpha+\beta)\|au\|\|bv\|}\\
=&\begin{cases}
\arccos \frac{\rho_{\alpha,\beta}(u,v)}{(\alpha+\beta)\|u\|\|v\|},~\text{ if~} ab > 0;\\
\arccos (-\frac{\rho_{\beta, \alpha}(u,v)}{(\alpha+\beta)\|u\|\|v\|}),~\text{ if~} ab < 0
\end{cases}\\
=&\begin{cases}
\theta_{\alpha,\beta}(u, v),~\text{ if~} ab > 0;\\
\pi-\theta_{\beta, \alpha}(u, v),~\text{ if~} ab < 0
\end{cases}
\end{align*}
\end{itemize}
\end{proof}
\begin{defn}\label{lm}
 Two norms, $\|.\|_1$ and $\|.\|_2$, on $X$ have the $\alpha, \beta$-angularly property
if there exists a constant $K$ such that for all non-zero elements $u, v \in X$,
\begin{align*}
\tan\Big(\frac{\theta_{\alpha,\beta,2}(u,v)}{2}\Big)\leq K\tan\Big(\frac{\theta_{\alpha,\beta,1}(u,v)}{2}\Big).
\end{align*}
Here $\theta_{\alpha,\beta,1}(u, v)$ and $\theta_{\alpha,\beta,2}(u, v)$ are the $\alpha,\beta$-angles from $u$ to $v$ relative to $\|.\|_1$ and $\|.\|_2$, respectively.  
\end{defn}

In a normed linear space $(X,\|.\|)$ is strictly convex (rotund) if and only
if $u=v$ and $\|u\|=\|v\|= 1$ together imply that $\|tu+(1-t)v\|< 1$ for all $0<t<1$.
In the following theorem we show that $\alpha,\beta$-angularly property of norms share a geometric property.

\begin{thm}
 Suppose the norms $\|.\|_1$ and $\|.\|_2$ have the $\alpha, \beta$-angularly property
on $X$. Then the following statements are equivalent:
\begin{itemize}
\item[(i)] $(X, \|.\|_1)$ is strictly convex.
\item[(ii)] $(X, \|.\|_2)$ is strictly convex.
\end{itemize}  
\end{thm}
\begin{proof}
First we show that $(i)$ implies $(ii)$.\\
A normed linear space is considered strictly convex if each boundary point of the unit ball is an extreme point. Therefore, it is sufficient to show that if $\frac{u}{\|u\|_1}$ is an extreme point of the $\|.\|_1$-unit ball, then $\frac{u}{\|u\|_2}$ is an extreme point of the $\|.\|_2$-unit ball. Let us consider that $\frac{u}{\|u\|_2}$
is not an extreme point of the $\|.\|_2$-unit ball. Then there are points $v$ and $w$ in $X$ such that $\frac{u}{\|u\|_2}=\frac{v+w}{2}$ and the closed line segment from
$v$ to $w$ is contained in the $\|.\|_2$-unit ball. If $s\in [0,1]$ then the points $(1-s)v + sw$
and $sv +(1-s)w$ are on the line segment and hence in the $\|.\|_2$-unit ball. Thus,
\begin{align*}
 2 = \|v+w\|_2 =&\|(1-s)v + sw + sv + (1 - s)w\|_2\\
\leq&\|(1-s)v + sw\|_2 +\|sv+(1-s)w\|_2 \\
\leq& 1 + 1 = 2.   
\end{align*}
It follows that $\|(1-s)v + sw\|_2 = \|sv + (1 - s)w\|_2 = 1$. In particular, we observe
that $\|v\|_2=\|w\|_2=1$.
Hence
\begin{align*}
\rho_{\alpha,\beta,2}(v, w)=&\alpha\rho_{-,2}(v,w)+\beta\rho_{+,2}(v,w)\\
=&\alpha\|v\|_2\lim_{t \to 0^-}\frac{\|v+tw\|_2-\|v\|_2}{t}+\beta\|v\|_2\lim_{t \to 0^+}\frac{\|v+tw\|_2-\|v\|_2}{t}\\
=&\alpha\lim_{s \to 0^-}\frac{\|v+\frac{s}{1-s}w\|_2-1}{\frac{s}{1-s}}+\beta\lim_{s \to 0^+}\frac{\|v+\frac{s}{1-s}w\|_2-1}{\frac{s}{1-s}}\\
=&\alpha\lim_{s \to 0^-}\frac{\|(1-s)v+sw\|_2-(1-s)}{s}+\beta\lim_{s \to 0^+}\frac{\|(1-s)v+sw\|_2-(1-s)}{s}\\
=&\alpha\lim_{s \to 0^-}\frac{1-(1-s)}{s}+\beta\lim_{s \to 0^+}\frac{1-(1-s)}{s}\\
=&\alpha+\beta.
\end{align*}
It follows that $\rho_{\alpha,\beta,2}(v, w)= \alpha+\beta$, $\cos(\theta_{\alpha,\beta,2}(v, w))= 1$, and $\tan(\theta_{\alpha,\beta,2})(v,w)=0$. By the
$\alpha,\beta$-angularly property, $\tan(\theta_{\alpha,\beta,1})(v,w)=0$ as well. This implies $\cos(\theta_{\alpha,\beta,1}(v,w))= 1$
and hence $\rho_{\alpha,\beta,1}(v,w)=(\alpha+\beta)\|v\|_1\|w\|_1$.
Now, 
\begin{align*}
(\alpha+\beta)\|v\|_1\|w\|_1&=
\rho_{\alpha,\beta,1}(v,w)\\
&=\rho_{\alpha,\beta,1}(v,v+w-v) \\
&=\rho_{\alpha,\beta,1}(v,v+w)+\rho_{\alpha, \beta,1}(v,-v)\\
 &=\rho_{\alpha,\beta,1}(v,v+w)-\rho_{\beta,\alpha,1}(v,v)\\
 &\leq (\alpha+\beta)\|v\|_1\|v+w\|_1-(\alpha+\beta)\|v\|_{1}^2\\
 &\leq (\alpha+\beta)(\|v\|_1\|v+w\|_1-\|v\|_{1}^2)\\
 &\leq (\alpha+\beta)\|v\|_1(\|v+w\|_1-\|v\|_1)\\
 &\leq (\alpha+\beta)\|v\|_1\|w\|_1.
\end{align*}
and hence
$(\alpha+\beta)\|v\|_1(\|v+w\|_1-\|v\|_1)=\|v\|_1\|w\|_1$. This implies that $\|v+w\|_1=\|v\|_1+\|w\|_1$.\\
Also,
\begin{align*}
\frac{u}{\|u\|_2}=\frac{\frac{v+w}{2}\|x\|_2}{\|\frac{v+w}{2}\|x\|_2\|_1}=\frac{v+w}{\|v+w\|_1}=\frac{\|v\|_1}{\|v\|_1+\|w\|_1}\frac{v}{\|v\|_1}+\frac{\|w\|_1}{\|v\|_1+\|w\|_1}\frac{w}{\|w\|_1}
\end{align*}
which is a convex combination of the points $\frac{v}{\|v\|_1}$ and $\frac{w}{\|w\|_1}$. Therefore, $\frac{u}{\|u\|_1}$ is an interior point of the line segment from $\frac{v}{\|v\|_1}$ to $\frac{w}{\|w\|_1}$. The convexity of this line segment demonstrates that the entire line segment lies in the $\|.\|_1$-unit ball. A contradiction follows from the fact that $\frac{u}{\|u\|_1}$
is not an extreme point of the $\|.\|_1$-unit ball.\\
Using a similar argument, we get $(ii)$ implies $(i)$.
\end{proof}

\section{Conclusion}
This article defines the $\rho_{\alpha,\beta}$ orthogonality as a linear combination of norm derivatives and explores its several intriguing geometric characteristics. Furthermore, we provide several characterizations of smooth normed spaces founded on the concept of $\rho_{\alpha,\beta}$-orthogonality. Also, we examine the connection between linear preserving mapping and $\rho_{\alpha,\beta}$-orthogonality. Is it possible to find a more general class of orthogonality than what is defined in this paper, with interesting geometric properties?

\textbf{Declarations }

\textbf{ Availability of data and materials :}
Data sharing not applicable to this article as no datasets were generated or analysed during the current study.

\textbf{Competing Interests :} The authors declare that they have no competing interests. 
 
 \textbf{Authors' contributions :}
All authors contribute equally.



\end{document}